\documentclass[11p,leqno]{amsart}
\usepackage{amsmath}
\usepackage{amsfonts}
\usepackage{amssymb}
\usepackage{graphicx}
\usepackage{color}
\usepackage{amsfonts,mathtools,mathrsfs}
\usepackage{amscd,esint}
\usepackage{amsmath,amsthm}
\usepackage{hyperref}
\usepackage{url}
\usepackage[backrefs]{amsrefs}
\usepackage{color}
\newtheorem{theorem}{Theorem}[section]
\newtheorem*{theorem*}{Theorem}
\newtheorem*{remark*}{Remark}

\newtheorem{lemma}{Lemma}[section]
\newtheorem{corollary}[theorem]{Corollary}
\newtheorem{proposition}{Proposition}[section]

\newtheorem{remark}[theorem]{Remark}

\numberwithin{equation}{section}

\begin{document}
\title{Parabolic frequency monotonicity on compact manifolds}
\author{Xiaolong Li}
\address{Department of Mathematics, University of California, Irvine, Irvine, CA 92697, USA}
\email{xiaolol1@uci.edu}
%
%
\author{Kui Wang}
%
%
\address{School of Mathematical Sciences, Soochow University, Suzhou, 215006, China}
\email{kuiwang@suda.edu.cn}
\maketitle
\begin{abstract}
This work is devoted to the study of parabolic frequency for solutions of the heat equation on Riemannian manifolds.
We show that the parabolic frequency functional is almost increasing on compact manifolds with nonnegative sectional curvature, which generalizes a monotonicity result proved by C. Poon \cite{Po96} and by L. Ni \cite{Ni15}. The proof is based on
a generalization of R. Hamilton's matrix Harnack inequality \cite{Ham1} for small time. As applications, we obtain a unique continuation result.   Monotonicity of a new quantity under two-dimensional Ricci flow, closely related to the parabolic frequency functional, is derived as well.
\end{abstract}

\section{Introduction}
The (elliptic) frequency functional for a harmonic function $h(x)$ on $\mathbb{R}^n$, introduced by F. J. Almgren \cite{AL79} in 1979 and used in the study of local regularity of (multiple-valued) harmonic functions and minimal surfaces, is defined by
$$
I_e ( r)=\frac{r \int_{B(o,r)} |\nabla h|^2\ d\mu}{\int_{\partial B(o,r)}  h^2\ dA},
$$
where $dA$ is the induced $n-1$ dimensional Hausdorff measure on $\partial B(o,r)$ and $o$ is a fixed point in $\mathbb{R}^n$. Almgren obeserved that $I_e(r)$ is monotone nondecreasing in $r$. For $n=2$, it was in fact first proved by G. H. Hardy using a complex analysis argument (see Exercise 6 on page 138 of \cite{Conway78}).
The monotonicity of $I_e (r)$ has many applications in partial differential equations and geometric measure theory. For instance, it was used by N. Garofalo and F.H. Lin  \cites{GaLin86, GaLin87} and F.H. Lin \cite{Lin91} to study the unique continuation properties for elliptic operators and to estimate the size of nodal sets of solutions to parabolic and ellptic equations. The frequency functional $I_e (r)$ also controls the vanishing order of harmonic functions at the center $o$, see the book \cite{HL}. We refer the readers to \cite{HL} and \cite{Ste08} for more applications.

For harmonic functions on Riemannian manifolds, N. Garofalo and F.H. Lin \cite{GaLin86} proved that $I_e (r)$ is almost increasing in the sense that
\emph{there exist constants $R$ and $\Lambda$, depending only on the Riemannian metric, such that $e^{\Lambda r} I_e(r)$ is monotone nondecreasing in $(0, R)$}, (see also \cite[Theorem 2.2]{Dan13}). More recently,  A. Logunov \cite{Log}\cite{Log16} used this almost monotonicity together with combinatorics techniques to estimate the size of nodal sets for harmonic functions and eigenfunctions on manifolds, and proved
Nadirashvili's conjecture, the lower bound in Yau's conjecture, and polynomial upper estimates of the Hausdorff measure of nodal sets of Laplace eigenfuctions.

The parabolic frequency functional for solutions of heat equation on $\mathbb{R}^n$  was introduced by C. Poon \cite{Po96} in 1996 and used in the study of the unique continuation of solutions to parabolic equations. We recall its definition for solutions of heat equation on Riemannian manifolds. Let $(M^m, g)$ be a complete Riemannian manifold, $o$ be a fixed point in $M$, and $d \mu$ be the volume element with respect to the Riemannian metric $g$.
Let $u(x,t)$ be a smooth nonconstant solution to the heat equation
\begin{equation}\label{heatequation}
u_t-\triangle_g u=0
\end{equation}
in $M\times [0,T]$.
Let  $H(x,o; t)$ be the fundamental solution to the heat equation (\ref{heatequation}), written as $H(x,t)$ for short. Assume either $M$ is compact or $M$ is complete with bounded geometry and $u(x,t)$ satisfies certain growth conditions so that the integrals are finite and all integration by parts can be justified. Then the parabolic frequency for $u$ is defined as
$$
I(t)=t\cdot \frac{\int_M H(x,t)\cdot |\nabla u|^2(x, T-t)\ d\mu}{\int_M H(x,t)\cdot u^2(x, T-t)\ d\mu}.
$$
It was shown by C. Poon \cite{Po96} and L. Ni \cite{Ni15} that if $M$ has nonnegative sectional curvature and parallel Ricci curvature, then $I(t)$ is monotone nondecreasing in $t$. The main ingredient of their proofs is the matrix Harnack estimate of R. Hamilton \cite{Ham1}, which asserts that on a Riemannian manifold with nonnegative sectional curvature and parallel Ricci curvature, the fundamental solution $H(x,t)$ satisfies
$$\nabla_i \nabla_j H -\frac{\nabla_i H \nabla_j H }{H} +\frac{H}{2t}g_{ij} \geq 0.$$
In fact, R. Hamilton proved the above matrix Harnack estimate for any positive solution of the heat equation. For K\"{a}hler manifolds with nonnegative bisectional curvature, L. Ni \cite{Ni15} also proved the monotonicity of $I(t)$ when $u$ is a holomorphic function. The proof again relies on a matrix Li-Yau-Hamilton estimate for solutions to the heat equation on K\"ahler manifolds that was established in \cites{CN05}\cites{Ni07}.

The parallel Ricci curvature assumption seems quite restrictive and it is our purpose of this paper to study parabolic frequency functional for solutions of heat equation on more general Riemannian manifolds. In particular, we prove the almost monotonicity of the parabolic frequency functional for a short time on compact manifolds with nonnegative sectional curvature. As applications, we obtain a unique continuation result (see corollary \ref{Co} below) for solutions of the heat equation on such manifolds. The main result of this paper is: 
\begin{theorem}
Let $(M^m,g)$ be a compact Riemannian manifold with nonnegative sectional curvature.
Assume $u(x,t)$ is a nonconstant solution to the heat equation (\ref{heatequation}) with the initial data $u_0(x)$.
There exists a constant $T>0$, depending on the manifold $M$ and $u_0(x)$,  such that
$$
e^{t^{1/2}}\cdot t\cdot\frac{\int_M H(x,t)\cdot |\nabla u|^2(x, T-t)\ d\mu}{\int_M H(x,t)\cdot u^2(x, T-t)\ d\mu}
$$
is monotone increasing in $t$ for $[0,T]$.
\end{theorem}

It is also natural to consider the case when the metric $g$ is evolving by a geometric flow. In this direction, we consider $(M^2, g(t) )$, a solution to Ricci flow on surfaces  with positive scalar curvature, and define a quantity $J(t)$ by
\begin{eqnarray}\label{J(t)}
  J(t)= t\cdot \frac{\int_M |\nabla v(x,t)|^2 \cdot R(x,t)\;d\mu_{g(t)}}{\int_M v^2(x,t)\cdot R(x,t)\;d\mu_{g(t)}},
\end{eqnarray}
where $v(x,t)$ is a solution of the backward heat equation and $R(x,t)$ is the scalar curvature.  We prove that $J(t)$ is monotone increasing in $t$.
\begin{theorem}\label{Monotonicity Theorem}
Let $M^2$ be a closed surface. Suppose that $g(t)$ is a solution to the Ricci flow $\frac{\partial}{\partial t} g =-R g$ on $M$ with positive scalar curvature for $t\in [0,T)$. Let $v(x, t)$ be a nonconstant solution to the backward heat equation
$v_t(x, t)+\Delta_{g(t)} v(x, t)=0$
on $M\times[0,T)$.
Then $J(t)$ defined in \eqref{J(t)} is monotone increasing in $t$ on $[0,T)$.
\end{theorem}
It would be desirable to find applications of this monotonicity formula and to extend this result to higher dimensions and to other geometric flows.


\section{Hamilton's matrix Harnack inequality for small time}
In this section, we present an improved version of Hamilton's matrix Harnack inequality  \cite{Ham1} for small time, which will be used in the proof of monotonicity of parabolic frequency on compact manifolds. We prove the following theorem.
\begin{theorem}\label{thm-harnack}
Let ($M^m, g$) be a compact Riemannian manifold with $\operatorname{Sect}_g\ge -K$, $K\ge0$ and $|\nabla \operatorname{Ric}|\le L$, and
$f(x,t)$ be a positive solution to the heat equation (\ref{heatequation}). Then for any $\epsilon>0$,
there exist constants $B=B(M)$ and $T=T(\epsilon, K, L, m)$ such that
$$
\nabla^2 \log f(x,t)+\frac{1}{2t}g\ge -\Big((\frac{34}{3}+\epsilon)K+\epsilon\Big)\Big(m+\log\frac{B}{t^{m/2}f
}\Big) g
$$
for $t\in(0,T]$.
\end{theorem}
To begin with,  we collect some well-known estimates on positive solutions to the heat equation,  which will be
used in the proof of Theorem \ref{thm-harnack}.
\begin{lemma}[Corollary 1.2, 1.3 and 4.2 in \cite{Ham1}]
Let $(M^m, g)$ be a compact Riemannian manifold with $\operatorname{Ric}_g\ge -(m-1)K$.
Suppose that $f(x,t)$ is a positive solution the heat equation (\ref{heatequation}), satisfying
$\int_M f\ d\mu\le 1$. There exists a constant $B=B(M)$ depending only on $M$ such that
\begin{eqnarray}
f(x,t)&\le& \frac{B}{t^{m/2}}\int_M f(x,t)\ d\mu,\\
t|\nabla f|^2&\le& \Big(2+2(m-1)Kt\Big)f^2 \log \frac{B}{t^{m/2}f},\label{grad}\\
\frac t 2 \triangle f&\le& \frac{4(m-1)Kt/2}{1-e^{-(m-1)Kt/2}}  \Big(m+\log \frac{B}{t^{m/2}f}\Big) f, \label{laplace}
\end{eqnarray}
for $0<t\le1$.
\end{lemma}

\begin{proof}[Proof of Theorem \ref{thm-harnack}]
The proof of Theorem \ref{thm-harnack} is essentially based on the computations from Hamilton's paper \cite{Ham1}.
Let
\begin{eqnarray*}
A&=&m+\log \frac{B}{t^{m/2}f},\\
P&=&\frac{f}{2t}+C A f,\\
H_{ij}&=&\nabla_i\nabla_j f-\frac{\nabla_i f \nabla_j f}{f},\\
N_{ij}&=&H_{ij}+Pg_{ij},\\
Q&=&-\frac{f}{2t^2}-\frac{m}{2t}C f+C \frac{|\nabla f|^2}{f},
\end{eqnarray*}
where $C$ is a constant to be specified later.
Denote by
$$
W_{ikjl}=R_{ikjl}+K(g_{ij}g_{kl}-g_{il}g_{jk}),
$$
then
$$
R_{ikjl}N_{kl}=W_{ikjl}N_{kl}-KNg_{ij}+KN_{ij},
$$
and $W_{ikjl}\ge 0$ by $\operatorname{Sect}_g\ge -K$, which implies
$$
R_{ikjl}\nabla_kf\nabla_lf\ge -K|\nabla f|^2 g_{ij}.
$$
Direct computations give
\begin{eqnarray*}
\frac{\partial}{\partial t} N_{ij}&=&\triangle N_{ij}+\frac 2 f N^2_{ij}-\frac 4 f P N_{ij}
+2R_{ijkl}N_{kl}-R_{ik}N_{jk}-R_{jk}N_{ik}\\
&&+\frac 2 f R_{ikjl} \nabla_k f\nabla_l f+\big(\nabla_l R_{ij}-\nabla_i R_{jl}-\nabla_j R_{il}\big)\nabla_l f+\big(\frac 2 f P^2+Q\big)g_{ij}.
\end{eqnarray*}
Since
$$
\Big(\nabla_l R_{ij}-\nabla_i R_{jl}-\nabla_j R_{il}\Big)\nabla_l f\ge -3L|\nabla f| g_{ij},
$$
we conclude
\begin{eqnarray*}
\frac{\partial}{\partial t} N_{ij}&\ge&\triangle N_{ij}+\frac 2 f N^2_{ij}-\frac 4 f P N_{ij}+2W_{ijkl}N_{kl}\\
&&+2KN_{ij}-R_{ik}N_{jk}-R_{jk}N_{ik}+Z g_{ij},
\end{eqnarray*}
where
$$
Z=\frac 2 f P^2+Q-2KN-\frac{2K}{f}|\nabla f|^2-3L|\nabla f|.
$$
Since $\frac{2}{f}P^2 =\frac{f}{2t^2}+\frac 2 t C A f+2C A^2 f $ and $A\ge m$,
we obtain
$$
\frac{2}{f}P^2+Q\ge\frac{3}{2t}C A f+C \frac{|\nabla f|^2}{f}+2C^2 A^2 f.
$$
Using the estimate (\ref{laplace}), we estimate
 \begin{eqnarray*}
2KN+2K\frac{|\nabla f|^2}{f}+3L|\nabla f|
&=&2K\triangle f+\frac{Km}{t}f+2KmC A f+3L |\nabla f|\\
&\le&\Big(\frac{16(m-1)Kt/2}{1-e^{-(m-1)Kt/2}}+1\Big)\frac{K}{t}A f+2KC A^2 f+3L |\nabla f|.
\end{eqnarray*}
Assembling these estimates, we have
 \begin{eqnarray*}
Z\ge \frac{3C}{2t}A f+C \frac{|\nabla f|^2}{f}+2C^2 A^2 f-\Big(16\frac{(m-1)Kt/2}{1-e^{-(m-1)Kt/2}}+1\Big)\frac{K A f}{t} -2KCA^2 f-3L |\nabla f|.
\end{eqnarray*}

Now we choose $C=(\frac{34}{3}+\epsilon)K+\epsilon$, and then
 \begin{eqnarray*}
Z&\ge& \frac{3\epsilon}{2t}A f+\Big(\frac{3\epsilon}{2}+16(1-\frac{(m-1)Kt/2}{1-e^{-(m-1)Kt/2}})\Big)\frac{K A f}{t}\\
&& +\Big((\frac{34}{3}+\epsilon)K+\epsilon\Big) \frac{|\nabla f|^2}{f}+2\Big((\frac{34}{3}+\epsilon)K+\epsilon\Big)^2 A^2 f\\
&&-2K\Big((\frac{34}{3}+\epsilon)K+\epsilon\Big)A^2 f-3L |\nabla f|\\
&\ge& \frac{3\epsilon}{2t}A f+\Big[\frac{3\epsilon}{2}+16(1-\frac{(m-1)Kt/2}{1-e^{-(m-1)Kt/2}})\Big]\frac{K A f}{t}-\frac{9}{4\epsilon}L^2 f,
\end{eqnarray*}
where we used
$$3L|\nabla f|\le \epsilon\frac{|\nabla f|^2}{f}+\frac{9}{4\epsilon}L^2 f.$$
Observing
$$
\lim_{t\rightarrow 0}(1-\frac{(m-1)Kt/2}{1-e^{-(m-1)Kt/2}})=0,
$$
then we conclude there exists a small $T$, depending on $K$, $m$ and $\epsilon$, such that
$$
\frac{3\epsilon}{2}+16(1-\frac{(m-1)Kt/2}{1-e^{-(m-1)Kt/2}})\ge 0,$$
for $t\le T$.
Furthermore if $t<2\epsilon^2m/(3L^2)$, we have
$$
\frac{3\epsilon}{2t}A f-\frac{9}{4\epsilon}L^2 f\ge 0.
$$
Thus we can choose a constant $T$, depending on $K$, $m$, $L$, and $\epsilon$, such that
 $$Z\ge 0.$$
Therefore, Hamilton's maximum principle for tensors implies
$$
N_{ij}\ge 0
$$
for $0<t\le T=T( K, m, L, \epsilon)$,
proving the theorem.
\end{proof}
Denote by $d(x)=d_g(x,o)$, the distance between $x$ and $o$.
For the  heat kernel $H(x,t)$, we have both upper bound and lower bound from Cheng-Li-Yau's paper \cite{CLY} and Hamilton's paper \cite{Ham1}. We summarize as the following lemma.
\begin{lemma}
Let $(M^m, g)$ be a compact Riemannian manifold. The fundamental solution $H(x,t)$ of heat equation on $M$ satisfies
\begin{equation}\label{upper}
H(x,t)\le \frac{C}{t^{m/2}}\exp{\Big(-\frac{d^2(x)}{5t}\Big)},
\end{equation}
and
\begin{equation}\label{lowerbound}
H(x,t)\ge \frac{C}{t^{m/2}}\exp{\Big(-\frac{d^2(x)}{4t}(1+2(m-1)Kt)-\frac m 2 e^{2(m-1)Kt}\Big)}
\end{equation}
for some constant $C$ depending  on $M$.
\end{lemma}

From the lower bound (\ref{lowerbound}), we deduce
\begin{eqnarray*}
A&=&m+\log\frac{B}{t^{m/2}H}\\
&\le& m+\log B-\log(Ce^{-\frac{d^2}{4t}(1+2(m-1)K t)-\frac m 2 e^{2(m-1)Kt}})\\
&\le&C_0+\frac{d^2}{4t}+\frac 1 2(m-1)Kd^2(x).
\end{eqnarray*}
Where $C_0$ is a constant depending on $M$ only.

In conclusion,  we get the following Harnack inequality for the heat kernel on compact manifolds with nonnegative sectional curvature.
\begin{corollary}\label{cor2}
Assume $M$ is a compact manifold with nonnegative sectional curvature. Then for any
 $\epsilon>0$, there exist constants $T=T(M, \epsilon)$ and $C_0=C_0(M)$, such that
\begin{equation}\label{harnack}
\nabla^2 \log H(x,t)+\frac{1}{2t}g\ge-\epsilon\big(C_0+\frac{d^2(x)}{4t}\big)g,
\end{equation}
for $t\in (0,T]$.
\end{corollary}

\section{Monotonicity of parabolic frequency}
In this section, we prove that the parabolic frequency on compact manifolds with nonnegative sectional curvature is almost
 increasing.
\begin{theorem}\label{thm-sd}
Let $(M^m,g)$ be a compact Riemannian manifold with nonnegative sectional curvature.
Assume $u(x,t)$ is a solution to the heat equation (\ref{heatequation}) with the initial data $u_0(x)$ satisfying
$$|u_0(x)|+|\nabla u_0(x)|+|\nabla^2 u_0(x)|\le a_0\cdot(\int_M |\nabla u_0|^2 \ d\mu)^{\frac 1 2},$$
for some positive $a_0$.
There exists a constant $T>0$, depending on the manifold $M$ and $a_0$  such that
$$
e^{t^{1/2}}\cdot t\cdot\frac{\int_M H(x,t)\cdot |\nabla u|^2(x, T-t)\ d\mu}{\int_M H(x,t)\cdot u^2(x, T-t)\ d\mu}
$$
is an increasing function of $t$ in $[0,T]$.
\end{theorem}
\begin{remark} The initial condition in Theorem \ref{thm-sd}
$$|u_0(x)|+|\nabla u_0(x)|+|\nabla^2 u_0(x)|\le a_0\cdot(\int_M |\nabla u_0|^2 \ d\mu)^{\frac 1 2}$$
is equivalent to that $u_0(x)$ is not a constant. In fact, if $u_0=constant$,  then $u(x,t)$ is also a constant and Theorem \ref{thm-sd} is trivial. \end{remark}
To begin with, we define the following quantities:
\begin{eqnarray*}
Z(t)&=&\int_M H(x,t)\cdot u^2(x, \tau)\ d\mu,\\
D(t)&=&\int_M H(x,t)\cdot |\nabla u|^2(x, \tau)\ d\mu,
\end{eqnarray*}
where $\tau=T-t$.
We first show the monotonicity  of quantities $D(t)$ and $Z(t)$,  by calculating their evolution equations.
\begin{proposition}\label{thm-fd} For $t\in[0,T]$, it holds that
\begin{equation}
Z'(t)=2D(t)>0.
\end{equation}
Assume further that $\operatorname{Ric}_g\ge-(m-1)K$, then
\begin{equation}
D'(t)\ge -2(m-1)K D(t),
\end{equation}
and therefore $e^{2(m-1)Kt}D(t)$ is monotone increasing.
\end{proposition}
\begin{proof}
Direct calculations show that
\begin{eqnarray}\label{dZ}
Z'(t)&=&\int_M H_t\cdot u^2\ d\mu-\int_M 2Hu\cdot u_\tau\ d\mu\\
&=&\int_M H\cdot \triangle u^2\ d\mu-\int_M 2Hu\cdot u_\tau \ d\mu\nonumber\\
&=&2\int_M H\cdot |\nabla u|^2 \ d\mu\nonumber\\
&=&2D(t)\nonumber,
\end{eqnarray}
which is clearly positive.

We compute that
\begin{eqnarray}\label{dD}
D'(t)&=&\int_M H_t |\nabla u|^2 -2H\langle \nabla u, \nabla u_\tau\rangle\ d\mu \\
&=&\int_M \triangle H |\nabla u|^2 -2H\langle \nabla u, \nabla u_\tau\rangle\ d\mu\nonumber\\
&=&\int_M  2H \langle \nabla u, \nabla \triangle  u\rangle+2H|\nabla^2 u|^2\ d\mu\nonumber\\
&&+\int_M 2H\langle \nabla u, \operatorname{Ric}(\nabla u)\rangle -2H\langle \nabla u, \nabla u_\tau\rangle\ d\mu\nonumber\\
&=&\int_M  2H |\nabla^2 u|^2 \ d\mu+2\int_M  H \langle \nabla u, \operatorname{Ric}(\nabla u)\rangle\ d\mu\nonumber .
\end{eqnarray}
Since the Ricci curvature is bounded from below by $-(m-1)K$, then
\begin{eqnarray*}
D'(t)&\ge& \int_M  2H |\nabla^2 u|^2 \ d\mu-2(m-1)K\int_M  H |\nabla u|^2\ d\mu \\
&\ge&-2(m-1)K D(t),
\end{eqnarray*}
completing the proof.
\end{proof}

To prove Theorem \ref{thm-sd}, we derive a lower bound on $D(t)$ first.
\begin{lemma}\label{pro}
Let $(M,g)$ and $u(x,t)$ be same as in Theorem \ref{thm-sd}. Assume further that $\int_M |\nabla u_0|^2 \ d\mu=1$. Then for any $\epsilon>0$, there exist constants $T=T(M,\epsilon)$, $C_M=C(M)$, $c=c(M, \epsilon)$ and $C=C(M, \epsilon, a_0)$, such that
\begin{equation}\label{eD}
D(t)\ge c e^{-C\cdot t^{-C_M \epsilon}}
\end{equation}
for $t\in[0,T]$.
\end{lemma}
\begin{proof}
Recall from  Corollary \ref{cor2} that for any $\epsilon>0$, there exists a constant $T=T(M, \epsilon)$ and $C_0=C_0(M)$, such that
\begin{equation}\label{Hharnack}
\nabla^2 H-\frac{\nabla H\otimes \nabla H}{H}+\frac{H}{2t}g\ge-\epsilon(C_0+\frac{d^2(x)}{4t})Hg,
\end{equation}
for $t\in (0,T]$.

Let $X=\nabla u$, and then it follows from the  heat equation (\ref{heatequation}) that
\begin{equation}\label{heatequationX}
X_t-\triangle X= -\operatorname{Ric}(X).
\end{equation}
Let
$$
W(t)= \int_M H(x,t)\cdot|\nabla X|^2(x, \tau)\ d\mu,
$$
then we see from (\ref{dD}) that
\begin{eqnarray}\label{dD1}
D'(t)=2W(t)+2\int_M  H \big\langle X, \operatorname{Ric}(X)\big\rangle\  d\mu.
\end{eqnarray}
Direct computations show that
\begin{eqnarray*}
W'(t)&=&\int_M H_t |\nabla X|^2 -2H\langle \nabla X, \nabla X_\tau\rangle\ d\mu\\
&=&\int_M \triangle H |\nabla X|^2 +2H\langle \triangle X, X_\tau\rangle+2H\langle \nabla_{\nabla \log H}X, X_\tau\rangle\ d\mu\\
&=&\int_M \triangle H |\nabla X|^2 +2H\langle X_\tau+\operatorname{Ric}(X), X_\tau\rangle+2H\langle \nabla_{\nabla \log H}X, X_\tau\rangle\ d\mu,
\end{eqnarray*}
where we used equation (\ref{heatequationX}).
Using integration by parts, we get
\begin{eqnarray}\label{D-1}
\\
\int_M \triangle H |\nabla X|^2 \ d\mu&=&-2\int_M H_i\langle X_j, X_{ji}\rangle\ d\mu\nonumber \\
&=&-2\int_M H_i\langle X_j, X_{ij}\rangle\ d\mu+2\int_M H_i R_{ijkl} u_{kj}u_l\ d\mu\nonumber\\
&=&2\int_M H_{ij}\langle X_j, X_i\rangle\ d\mu+2\int_M H\langle \nabla_{\nabla \log H}X, \triangle X\rangle\ d\mu\nonumber\\
&&+ 2\int_M H_i R_{ijkl} u_{kj}u_l\ d\mu\nonumber\\
&=&2\int_M H_{ij}\langle X_j, X_i\rangle\ d\mu+2\int_M H\langle \nabla_{\nabla \log H}X, X_\tau+\operatorname{Ric}(X)\rangle\ d\mu\nonumber\\
&&+ 2\int_M H_i R_{ijkl} u_{kj}u_l\ d\mu.\nonumber
\end{eqnarray}
Recall from Harnack inequality (\ref{Hharnack}) that
$$
H_{ij}-\frac{H_iH_j}{H}+\frac{H}{2t}g_{ij}\ge -E H g_{ij},
$$
with $E=\epsilon(C_0+\frac{d^2(x)}{4t})$, we have
\begin{equation}\label{D-2}
2\int_M H_{ij}\langle X_j, X_i\rangle\ d\mu\ge 2 \int_M  H |\nabla_{\nabla \log H}X|^2\ d\mu-\frac{W(t)}{t}-2\int_M E H |\nabla X|^2\ d\mu.
\end{equation}

Now we estimate the curvature involved term  by
\begin{equation}\label{D-3}
2\int_M H_i R_{ijkl} u_{kj}u_l\ dx\ge-C_M\int_M |\nabla H|\cdot|X|\cdot|\nabla X|\ d\mu.
\end{equation}
Here and thereafter, we always use $C_M$ to denote the constant depending only on the manifold, though it may change from line to line. Combining above calculations and estimates (\ref{D-1}--\ref{D-3}) together, we conclude
\begin{eqnarray*}
W'(t)&\ge&\int_M 2 H |\nabla_{\nabla \log H}X|^2\ d\mu-\frac{W(t)}{t}-2\int_M E H |\nabla X|^2\ d\mu\\
&&+\int_M 2H\langle \nabla_{\nabla \log H}X, X_\tau+\operatorname{Ric}(X)\rangle\ dx-C_M\int_M |\nabla H|\cdot|X|\cdot|\nabla X|\ d\mu\\
&&+\int_M 2H\langle X_\tau+\operatorname{Ric}(X), X_\tau\rangle+2H\langle \nabla_{\nabla \log H}X, X_\tau\rangle\ d\mu\\
&=&\int_M 2 H \big|\nabla_{\nabla \log H}X+X_\tau\big|^2+2H\langle \operatorname{Ric}(X), \nabla_{\nabla \log H}X+X_\tau\rangle\ d\mu\\
&&-\frac{W(t)}{t}-2\int_M E H |\nabla X|^2\ d\mu-C_M\int_M |\nabla H|\cdot|X|\cdot|\nabla X|\ d\mu\\
&=&\int_M 2 H \big|\nabla_{\nabla \log H}X+X_\tau+\frac 1 2 \operatorname{Ric}(X)\big|^2\ d\mu-\frac 1 2\int_M H |\operatorname{Ric}(X)|^2\ d\mu\\
&&-\frac{W(t)}{t}-2\int_M E H |\nabla X|^2\ d\mu-C_M\int_M |\nabla H|\cdot|X|\cdot|\nabla X|\ d\mu,
\end{eqnarray*}
using the evolution equation (\ref{dD1}) for $D(t)$, we get
 \begin{eqnarray*}
W'(t)D(t)-W(t) D'(t)&\ge& \int_M 2 H |\nabla_{\nabla \log H}X+X_\tau+\frac 1 2 \operatorname{Ric}(X)|^2\ d\mu\cdot\int_M H |X|^2\ d\mu\\
&&-\frac 1 2\int_M H |\operatorname{Ric}(X)|^2\ d\mu\cdot D(t)-\frac{W(t)}{t}\cdot D(t)-2\int_M E H |\nabla X|^2\ d\mu\cdot D(t)\\
&&-C_M\int_M |\nabla H|\cdot|X|\cdot|\nabla X|\ d\mu\cdot D(t)-2W\cdot\Big(W+\int_M H\langle \operatorname{Ric}(X), X\rangle\ d\mu\Big).
\end{eqnarray*}
Integration by parts yields
\begin{eqnarray*}
\int_M H|\nabla X|^2\ d\mu&=&-\int_M H\langle \nabla_{\nabla \log H}X, X\rangle+H\langle X_\tau+\operatorname{Ric}(X), X \rangle\ d\mu\\
&=&-\int_M H\langle \nabla_{\nabla \log H}X+X_\tau, X\rangle\ d\mu-\int_M H\langle \operatorname{Ric}(X), X \rangle\ d\mu,
\end{eqnarray*}
then we have
\begin{eqnarray*}
2W\cdot\big(W+\int_M H\langle \operatorname{Ric}(X), X\rangle\ d\mu\big)&=&2\Big(\int_M  H \big\langle \nabla_{\nabla \log H}X+X_\tau+\frac 1 2 \operatorname{Ric}(X), X\big\rangle\ d\mu\Big)^2\nonumber \\
&&-\frac{1}{2}\Big(\int_M H \langle \operatorname{Ric}(X), X\rangle\ d\mu\Big)^2.
\end{eqnarray*}
Thus, using H\"older's inequality, we get
 \begin{eqnarray}\label{wd}
 W'D-W D'&\ge&-\frac 1 2\int_M H |\operatorname{Ric}(X)|^2\ dx\cdot D-\frac{W}{t}\cdot D\\
&&-2\int_M E H |\nabla X|^2\ d\mu\cdot D-C_M\int_M |\nabla H||X||\nabla X|\ d\mu\cdot D.\nonumber
\end{eqnarray}
Since the gradient estimates (\ref{grad}) gives
\begin{eqnarray*}
|\nabla H|\le \frac{\sqrt{2} H}{\sqrt{t}}\big(C_0+\frac{d^2}{4t}\big)^{1/2},
\end{eqnarray*}
we then estimate
\begin{eqnarray*}
C_M\int_M |\nabla H||X||\nabla X|\ d\mu &\le& C_M\int_M \frac{\sqrt{2}H}{\sqrt{t}}\big(C_M+\frac{d^2}{4t}\big)^{1/2}|X||\nabla X|\ d\mu\\
&\le&C_M \int_M\frac{1} {\epsilon t}H|X|^2\ d\mu+C_M\int_M H|\nabla X|^2\big(\epsilon+\frac{\epsilon }{t}\big)\ d\mu\\
&\le&\frac{C_M}{\epsilon t}D(t)+C_M (\epsilon+\frac{\epsilon}{t})W(t).
\end{eqnarray*}
Observing
\begin{eqnarray*}
2\int_M E H |\nabla X|^2\ d\mu \le C_M(\epsilon+\frac{\epsilon}{t}) W,
\end{eqnarray*}
and plugging above estimates to (\ref{wd}), we derive
\begin{eqnarray}
W'(t)D(t)-W(t) D'(t)&\ge&- C_M D^2-\frac{W D}{t}-C_M (\epsilon+\frac{\epsilon}{t})W D -\frac{C_M D^2}{t\epsilon}\\
&\ge&-\frac{C_M}{t\epsilon}D^2-\big(\frac{1+C_M \epsilon}{t}\big)W D\nonumber
\end{eqnarray}
for $t\le T$. Thus
$$
\Big(\frac{W(t)}{D(t)}\Big)'\ge-\frac{C_M}{t\epsilon}-\Big(\frac{1+C_M\epsilon}{t}\Big)\frac{W(t)}{D(t)},
$$
which implies
$$t^{1+C_M \epsilon}\frac{W(t)}{D(t)}+\frac{C_M}{(1+C_M\epsilon)\epsilon} t^{1+C_M \epsilon}$$
is monotone nondecreasing in $[0,T]$. Therefore
$$
\frac{W(t)}{D(t)}\le t^{-1-C_M \epsilon}\cdot C(M,  \epsilon, a_0),
$$
for $t\le T$.

From derivative of $D(t)$ in (\ref{dD}) and  above estimate, we obtain that
\begin{eqnarray*}
(\log D(t))'=\frac{2W(t)+2\int_M H\langle \operatorname{Ric}(X), X \rangle\ d\mu}{D(t)} \le t^{-1-C_M \epsilon}\cdot C(M, \epsilon, a_0).
\end{eqnarray*}
Integrating from $t$ to $T$ yields
$$
D(t)\ge D(T) e^{-C(M, \epsilon, a_0) \cdot t^{-C_M\cdot \epsilon}}\ge c(M, \epsilon)\cdot e^{-C(M, \epsilon, a_0) \cdot t^{-C_M\cdot \epsilon}}
$$
proving the lemma.
\end{proof}

Using the above lemma, we give the following estimate.
\begin{lemma}\label{lee}
Let $(M,g)$ and $u(x,t)$ be same as in Theorem \ref{thm-sd}. Assume $\int_M |\nabla u_0|^2 \ d\mu=1$. There exists a constant $T_0=T_0(M, a_0)$, such that
\begin{equation}\label{eeD}
\frac{1}{t^{1/2}}\int_M  H|\nabla u|^2 d^2(x)\ d\mu \le  \frac 3 2 D(t),
\end{equation}
for any $0<t\le T_0$.
\end{lemma}
\begin{proof}
We first choose $\epsilon$ in Lemma \ref{pro} so that
$$
C_M \epsilon =\frac 1 4 ,
$$
which clearly depends only on $M$. Observing that
\begin{eqnarray*}
\frac{1}{t^{1/2}}\int_M d^2H|\nabla u|^2\ d\mu&=&\int_{d^2(x)\le t^{1/2}}  \frac{d^2}{t^{1/2}} H|\nabla u|^2\ d\mu+\int_{d^2(x)\ge t^{1/2}}  \frac{d^2}{t^{1/2}} H|\nabla u|^2\ d\mu\\
&\le&D(t)+\int_{d^2(x)\ge t^{1/2}}  \frac{d^2}{t^{1/2}} H|\nabla u|^2\ d\mu,
\end{eqnarray*}
then we only need to estimate the integral  $$\int_{d^2(x)\ge t^{1/2}}  \frac{d^2}{t^{1/2}} H|\nabla u|^2\ d\mu.$$

From the upper bound of the heat kernel (\ref{upper}), it follows
\begin{eqnarray*}
\int_{d^2(x)\ge t^{1/2}}  \frac{d^2}{t^{1/2}} H|\nabla u|^2\ d\mu&\le&C(M)\int_{d^2(x)\ge t^{1/2}}  \frac{e^{-d^2/(5t)}}{t^{(m+1)/2}} |\nabla u|^2\ d\mu\\
&=&C(M)\frac{e^{-\frac {1}{5t^{1/2}}}}{t^{(m+1)/2}} \int_{d^2(x)\ge t^{1/2}}|\nabla u|^2  \ d\mu\\
&\le&C(M)e^{-\frac{t^{-\frac 1 2}}{10}} \int_{M}  |\nabla u|^2\ d\mu\\
&\le&C(M)e^{-\frac{t^{-\frac 1 2}}{10}}.
\end{eqnarray*}
From Lemma \ref{pro}
$$
D(t)\ge c(M)\cdot e^{-C(M, a_0) \cdot t^{-\frac{1}{4}}}
$$
 and
$$e^{-\frac{t^{-\frac 1 2}}{10}}=o(c e^{-C t^{-\frac 1 4}}),$$
then there exists a constant
$T_0$, depending on $M$ and $a_0$, such that
$$
\int_{d^2(x)\ge t^{1/2}}  \frac{d^2}{t^{1/2}} H|\nabla u|^2\ d\mu\le \frac 1 2 D(t),
$$
for $t\in(0, T_0]$,
proving the lemma.
 \end{proof}
 Now we are using above lemmas to prove Theorem \ref{thm-sd}.
\begin{proof}[Proof of Theorem \ref{thm-sd}]
Without loss of generality, to prove Theorem \ref{thm-sd}, we assume that
$$
\int_M |\nabla u_0|^2 \ d\mu=1.
$$
Let $T$ be a constant, depending on  $M$ and $a_0$, which is so defined that  Corollary \ref{cor2} holds with $\epsilon=1/2$, i.e.
$$
\nabla^2 H-\frac{\nabla H\otimes \nabla H}{H}+\frac{H}{2t}g\ge-\frac 1 2\big(C_0+\frac{d^2(x)}{4t}\big)Hg.$$
We assume further
$$T\le \min \Big\{T_0, \frac{1}{64C_0^2}\Big\}$$
so that the estimate (\ref{eeD}) holds as $t\le T$.

Noting
\begin{eqnarray*}
D'(t)&=&\int_M H_t\cdot |\nabla u|^2\ d\mu-\int_M 2H\langle \nabla u, \nabla u_\tau \rangle\ d\mu\\
&=&\int_M \triangle H|\nabla u|^2\ d\mu+\int_M 2Hu_\tau^2\ d\mu+\int_M 2\langle\nabla H, \nabla u\rangle u_\tau\ d\mu\\
&=&\int_M 2 H_{ij}u_i u_j\ d\mu+\int_M 2Hu_\tau^2\ d\mu+\int_M 4H\langle\nabla \log H, \nabla u\rangle u_\tau\ d\mu,
\end{eqnarray*}
and matrix Harnack inequality (\ref{harnack}) gives
\begin{eqnarray*}
\int_M 2 H_{ij}u_i u_j\ d\mu \ge \int_M \Big( 2 H|\langle\nabla \log H, \nabla u\rangle|^2-\frac 1 t H |\nabla u|^2-\big(C_0+\frac{d^2(x)}{4t}\big) H |\nabla u|^2\Big)\ d\mu,
\end{eqnarray*}
we obtain
\begin{eqnarray*}
D'(t)&\ge&\int_M 2 H|\langle\nabla \log H, \nabla u\rangle|^2\ d\mu-\frac{D(t)}{t}-\int_M \big(C_0+\frac{d^2(x)}{4t}\big) H |\nabla u|^2\ d\mu\\
&&+\int_M 4H\langle\nabla \log H, \nabla u\rangle u_\tau\ d\mu+\int_M 2Hu_\tau^2\ d\mu\\
&=&\int_M 2 H|\langle\nabla \log H, \nabla u\rangle+u_\tau|^2\ d\mu-\frac{D(t)}{t}-\int_M \big(C_0+\frac{d^2(x)}{4t}\big) H |\nabla u|^2\ d\mu.
\end{eqnarray*}
These imply
\begin{eqnarray*}
D'(t)Z(t)-D(t)Z'(t)&\ge&\int_M 2 H\big|\langle\nabla \log H, \nabla u\rangle+u_\tau\big|^2\ d\mu\cdot \int_M Hu^2\ d\mu
-\frac{D(t)Z(t)}{t}\\&&-\int_M \big(C_0+\frac{d^2(x)}{4t}\big) H |\nabla u|^2\ d\mu\cdot Z(t)-2D^2(t).
\end{eqnarray*}
Since
$$
D(t)=\int_M H |\nabla u|^2\ d\mu=-\int_M Hu\langle\nabla \log H, \nabla u\rangle+H u u_\tau\ d\mu,
$$
then
$$
\int_M 2 H|\langle\nabla \log H, \nabla u\rangle+u_\tau|^2\ d\mu\cdot \int_M Hu^2\ d\mu-2D^2(t)\ge 0
$$
by the H\"older inequality. Then we get
\begin{eqnarray}\label{Ze1}
\Big(\frac{D(t)}{Z(t)}\Big)'&=&\frac{1}{Z^2}\Big(D'(t)Z(t)-D(t)Z'(t)\Big)\\
&\ge&\frac{1}{Z^2}\Big(-\frac{D(t)Z(t)}{t}-\int_M \big(C_0+\frac{d^2(x)}{4t}\big) H |\nabla u|^2\ d\mu\cdot Z(t)\Big)\nonumber\\
&=&-\frac 1 t \frac {D(t)} {Z(t)}-C_0 \frac {D(t)} {Z(t)}-\frac{1}{ 4 t}\frac{\int_M d^2H|\nabla u|^2\ d\mu}{Z(t)}\nonumber
\end{eqnarray}
Using estimate (\ref{eeD}), we obtain from (\ref{Ze1}) that
\begin{eqnarray*}
\Big(\frac{D(t)}{Z(t)}\Big)'\ge-\frac 1 t \frac {D(t)} {Z(t)}-C_0 \frac {D(t)} {Z(t)}-\frac{3}{8 t^{1/2}}\frac {D(t)} {Z(t)}
\ge-\frac 1 t \frac {D(t)} {Z(t)}-\frac{ 1}{ 2 t^{\frac 1 2}}\frac {D(t)} {Z(t)}
\end{eqnarray*}
for $0<t\le T$, which implies
$$
\left(e^{t^{\frac 1 2}}\cdot t\cdot \frac{ D(t)}{Z(t)}\right)'\ge 0,
$$
proving the theorem.

\end{proof}
\begin{remark}
The conclusion in Theorem \ref{thm-sd} also holds for almost nonnegative manifolds with almost the same proof as Theorem \ref{thm-sd}'s, i.e.  sectional curvature nonnegative in Theorem \ref{thm-sd} can be replaced by $\operatorname{Sect}_g\ge -K$ with $K\operatorname{diam}^2(M)\le \epsilon_0$ for some small positive $\epsilon_0$.
\end{remark}

It is well-known that on $\mathbb{R}^n$ a unique continuation (or backward uniqueness) theorem  follows from the monotonicity of parabolic frequency, and then from Theorem \ref{thm-sd}, we conclude  the following backward uniqueness theorem for the heat equation on compact manifolds with nonnegative sectional curvature.
\begin{corollary}\label{Co}
Let $(M^m,g)$ be a compact Riemannian manifold with nonnegative sectional curvature.  Let $u(x,t)$ be a smooth solution to the heat equation (\ref{heatequation}) in $M\times(0, +\infty)$. If $u(x,t)$ vanishes of infinite order in $(x_0, t_0)$ in the sense that
\begin{equation}
|u(x,t)|\le O\big(d^2(x,x_0)+|t-t_0|\big)^N
\end{equation}
for any integer $N>0$, for any $(x,t)$ near $(x_0, t_0)$. Then $u(x,t)$ is identically zero.
\end{corollary}
\begin{proof}
We assume by contradiction that
$$c_1:=\int_M H(x,x_0;t_0)\cdot |\nabla u|^2(x, 0)\ dx>0.$$
Assume further that
$$|u(x,t)|_{C^2}\le c_2,$$
for $M\times[0,t_0]$.

Let
$$
Z(t)=\int_M H(x,x_0;t)\cdot u^2(x, t_0-t)\ d\mu,
\text{\quad and\quad}D(t)=\int_M H(x,,x_0;t)\cdot |\nabla u|^2(x, t_0-t)\ d\mu.
$$
Then by Theorem \ref{thm-sd}, there exist a constant $T$ (depending on $M$ and $c_2/c_1$), such that
$$
e^{t^{1/2}}\cdot t\cdot\frac{D(t)}{Z(t)}
$$
is monotone nondecreasing in $[T-t_0,t_0]$. Let $e^{ t_0^{1/2}}\cdot t_0\cdot\frac{D(t_0)}{Z(t_0)}=C(t_0)$, and then it follows
$$
\big(\log Z(t)\big)'\le \frac{2C(t_0)}{t},
$$
which implies
\begin{equation}\label{Zes1}
Z(t)\ge Z(t_0)\big(\frac t {t_0}\big)^{2C(t_0)}
\end{equation}
for $0<t<t_0$. Here $Z(t_0)$ and $D(t_0)$ are nonzero due to the assumption on $D(T)$ and Proposition \ref{pro}.

But on the other hand if
$u$ vanishes of infinite order in $(x_0, t_0)$, then for any integer $N>0$ there exist constant $C_1>0$ and $\theta>0$, such that
for any $(x,t)$ satisfying $d^{2}(x,x_0)+|t-t_0|\le \theta$, it holds
$$
|u(x,t)|\le C_1\big(d^{2N}(x,x_0)+|t-t_0|^N\big).
$$
For any $t$ satisfying $t^{1/2}+t\le \theta$, we estimate
\begin{eqnarray*}
Z(t)=\int_M H(x,x_0;t)\cdot u^2(x, t_0-t)\ d\mu\le C_M \int_M u^2(x, t_0-t)\cdot t^{-\frac m 2}e^{-\frac{d^2(x,x_0)}{5t}}\ d\mu.
\end{eqnarray*}
Since
\begin{eqnarray*}
\int_{d\le t^{1/4}} u^2(x, t_0-t) t^{-\frac m 2}e^{-\frac{d^2(x,x_0)}{5t}}\ d\mu&\le&C(M, C_1)\int_{d\le t^{1/4}} t^{\frac N 2} t^{-\frac m 2}e^{-\frac{d^2(x,x_0)}{5t}}\ d\mu\\
&\le& C(M, C_1) t^{N/2-m/2}
\end{eqnarray*}
and
\begin{eqnarray*}
\int_{d>t^{1/4}} u^2(x, t_0-t)\cdot t^{-\frac m 2}e^{-\frac{d^2(x,x_0)}{5t}}\ d\mu&\le& C(c_2)\int_{d> t^{1/4}} t^{-\frac m 2}e^{-\frac{d^2(x,x_0)}{10t}}e^{-\frac{1}{10t^{1/2}}}\ d\mu\\
&\le& C(c_2) e^{-\frac{1}{10}t^{-1/2}}\\
&\le& C(c_2) t^{N/2-m/2}
\end{eqnarray*}
for $t$ small.  These give
\begin{eqnarray*}
Z(t)\le C t^{N/2-m/2}.
\end{eqnarray*}
Since $N$ is arbitrary large, the above inequality contradict with the estimate (\ref{Zes1}) as $t$ goes to zero.
Then we have $D(T)=0$, which immediately implies that $u(x,t)$ is identically zero for $t\in [t_0-T, t_0]$,
hence for all $t$.

\end{proof}

\section{A monotonicity formula for Ricci flow on surfaces}

In this section, we introduce a quantity $J(t)$ (see \eqref{J(t)}),
which is closely related to the parabolic frequency functional, and prove its monotonicity under Ricci flow on surfaces.

\begin{theorem}\label{Monotonicity Theorem}
Let $M^2$ be a closed surface. Suppose that $g(t)$ is a solution to the Ricci flow $\frac{\partial}{\partial t} g =-R g$ on $M$ with positive scalar curvature for $t\in [0,T)$. Let $v(x, t)$ be a nonconstant solution to the backward heat equation
$$v_t(x, t)+\Delta_{g(t)} v(x, t)=0$$
on $M\times[0,T)$.
Define
\begin{eqnarray}\label{J(t)}
  J(t)= t\cdot \frac{\int_M |\nabla v(x,t)|^2 \cdot R(x,t)\;d\mu_{g(t)}}{\int_M v^2(x,t)\cdot R(x,t)\;d\mu_{g(t)}}.
\end{eqnarray}
Then $J(t)$ is monotone increasing in $t$ on $[0,T)$.
\end{theorem}
The quantity $J(t)$
can be viewed as an entropy on two-dimensional Ricci flow. The crucial ingredient of the proof is a matrix differential Harnack estimate for Ricci flow on surfaces, which was obtained by R. Hamilton in the 1980's and was included in \cite[Exercise 10.22]{CLN}. It can be proved by applying Hamilton's maximum principle for tensors to the evolution equation satisfied by the quantity on the left hand side of \eqref{Harnack}.
\begin{lemma}
Let $(M^2, g(t))$ be a solution to the Ricci flow with positive scalar curvature for $t\in [0,T)$. Then for any $t\in (0,T)$, we have
\begin{equation}\label{Harnack}
    \nabla_i \nabla_j \log R + \frac{1}{2}\left(R+\frac{1}{t}\right) g_{ij} \geq 0.
\end{equation}
\end{lemma}

Before giving the proof of Theorem \ref{Monotonicity Theorem}, we recall some evolution equations for the Ricci flow on surfaces that can be found in \cite{CLN}.
\begin{lemma}\label{measure evolve}
Let $(M^2, g(t))$ be a solution to the Ricci flow $\frac{\partial}{\partial t} g =-R g$. Then we have
\begin{eqnarray*}
    \frac{\partial}{\partial t} g^{ij} &=&  R g^{ij}, \\
    \frac{\partial}{\partial t} (R\;d\mu) &=& \Delta R \;d\mu.
\end{eqnarray*}
\end{lemma}

\begin{proof}[Proof of Theorem \ref{Monotonicity Theorem}]
In the following, all the integrals are preformed with respect to $du_{g(t)}$, the Riemannian measure induced by the metric $g(t)$. To keep notations simple, we omit writing it.
Let
\begin{eqnarray*}
Z_1(t)&=& \int_M v^2(x,t) R(x,t)\ d\mu >0,  \\ D_1(t)&=& \int_M |\nabla v(x,t)|^2  R(x,t)\ d\mu.
\end{eqnarray*}
Direct calculation using Lemma \ref{measure evolve} shows
\begin{eqnarray*}
    Z_1'(t)&=& \int_M 2vv_t R\ d\mu + \int_M v^2 \Delta R \ d\mu \\
    &=& \int_M -2v \Delta v R \ d\mu+ \int_M (2v \Delta v + 2 |\nabla v|^2 ) R\ d\mu\\
    &=& 2 D_1(t).
\end{eqnarray*}
Making use of
$$\frac{\partial}{\partial t} |\nabla v|^2 =\left(\frac{\partial}{\partial t} g^{ij}\right) \nabla_i v \nabla_j v + 2g^{ij}\frac{\partial}{\partial t}(\nabla_i v ) \nabla_j v =|\nabla v|^2 R +2 \langle \nabla v, \nabla v_t \rangle, $$
and
$$\int_M |\nabla v|^2 \Delta R\ d\mu =-2\int_M \nabla_j R \nabla_j \nabla_i v \nabla_i v\ d\mu =2\int_M \nabla_i \nabla_j R \nabla_i v \nabla_j v \ d\mu-2 \int_M \langle \nabla R, \nabla v \rangle v_t \ d\mu,$$
yields
\begin{eqnarray*}
    D_1'(t)&=& \int_M |\nabla v|^2 R^2 \ d\mu+ 2\int_M \langle \nabla v, \nabla v_t \rangle R\ d\mu\\
    &&+ 2\int_M  \nabla_i \nabla_j R \nabla_i v \nabla_j v\ d\mu-2\int_M \langle \nabla R, \nabla v \rangle v_t \ d\mu\\
    &=& \int_M |\nabla v|^2 R^2 \ d\mu+ 2\int_M  v_t^2  R\ d\mu\\
    &&+ 2\int_M  \nabla_i \nabla_j R \nabla_i v \nabla_j v\ d\mu-4\int_M \langle \nabla R, \nabla v \rangle v_t\ d\mu.
\end{eqnarray*}
The matrix differential Harnack estimate in Lemma \ref{Harnack} is equivalent to
$$\nabla_i \nabla_j R -R\nabla_i \log R \nabla_j \log R + \frac{1}{2}R\left(R+\frac{1}{t}\right) g_{ij} \geq 0.$$
Substituting this into the expression for $D_1'(t)$ yields
\begin{eqnarray*}
    D_1'(t)&\geq & 2 \int_M R\langle \nabla v, \nabla \log R\rangle^2\ d\mu -\frac{1}{t} \int_M |\nabla v|^2 R\ d\mu\\
    &&+ 2\int_M  v_t^2  R \ d\mu-4\int_M R\langle \nabla \log R, \nabla v \rangle v_t \ d\mu\\
    &=& 2\int_M R \left| \langle \nabla v, \nabla \log R \rangle -v_t \right|^2 \ d\mu-\frac{D_1(t)}{t}.
\end{eqnarray*}
Therefore,
\begin{eqnarray*}
   D'_1(t)Z_1(t)+\frac{Z_1(t) D_1(t)}{t}
   &\geq & 2\int_M R \left| \langle \nabla v, \nabla \log R \rangle v_t \right|^2\ d\mu \cdot \int_M v^2 R \ d\mu\\
     &\geq&  2\Big(\int_M v R \left(v_t - \langle \nabla v, \nabla \log R \rangle\right)\ d\mu\Big)^2\\
     &\ge& 2D_1^2(t) = Z'_1(t) D_1(t),
\end{eqnarray*}
where we have used the Cauchy-Schwarz inequality, $Z'_1(t)=2D(t)$, and
\begin{eqnarray*}
D_1(t)& =& \int_M |\nabla v |^2 R \ d\mu= -\int_M \langle \nabla v, \nabla R \rangle v\ d\mu -\int_M v \Delta v R \ d\mu\\
&=& \int_M   v R \left(v_t - \langle \nabla v, \nabla \log R \rangle\right)\ d\mu. \end{eqnarray*}
Theorem \ref{Monotonicity Theorem} then follows immediately since
\begin{eqnarray*}
  J'(t) =\frac{t }{Z_1(t)^2}
  \left(\frac{Z_1(t)D_1(t)}{t}+D_1'(t) Z_1(t) -Z'_1(t) D_1(t)\right) \geq 0.
\end{eqnarray*}
This finishes the proof.
\end{proof}


The quantity $J(t)$ can also be viewed as a  Dirichlet energy with respect to the weighted evolving measure $R(x,t) d\mu$.
For any $0<t<T$, we define the first nonzero eigenvalue of $(M^2, g(t))$ with the weighted measure $R(x,t)d\mu$  by
\begin{eqnarray*}
\lambda_{R}(t)=\inf\left\{\frac{\int_M |\nabla u|^2  R(x,t)\;d\mu_{g(t)}}{\int_M u^2  R(x,t)\;d\mu_{g(t)}}: u(x)\in C^{\infty}(M)\setminus \{0\}, \int_M u(x) R(x,t) \ d\mu_{g(t)}=0\ \right\}.
\end{eqnarray*}
Then it is easy to see that the following corollary  is a direct consequence of Theorem \ref{Monotonicity Theorem}.
\begin{corollary}
Let $M^2$ be a closed surface. Suppose that $g(t)$ is a solution to the Ricci flow $\frac{\partial}{\partial t} g =-R g$ on $M$ with positive scalar curvature for $t\in[0,T)$, and $\lambda_{R}(t)$ is the eigenvalue defined as above. Then $t\lambda_{R}(t)$ is a monotone increasing  function of $t$ in $[0,T)$.
\end{corollary}

\section*{Acknowledgments}
The second author was supported by NSF of China under Grant No. 11601359, NSF of Jiangsu Province No. BK20160301, and China
Postdoctoral Foundation grant No. 2017T100394 and 2016M591900. 






\bibliographystyle{alpha}
\bibliography{pf}

\begin{bibdiv}
\begin{biblist}

\bib{AL79}{incollection}{
      author={Almgren, Frederick~J., Jr.},
       title={Dirichlet's problem for multiple valued functions and the
  regularity of mass minimizing integral currents},
        date={1979},
   booktitle={Minimal submanifolds and geodesics ({P}roc. {J}apan-{U}nited
  {S}tates {S}em., {T}okyo, 1977)},
   publisher={North-Holland, Amsterdam-New York},
       pages={1\ndash 6},
      review={\MR{574247}},
}

\bib{CN05}{article}{
      author={Cao, Huai-Dong},
      author={Ni, Lei},
       title={Matrix {L}i-{Y}au-{H}amilton estimates for the heat equation on
  {K}\"ahler manifolds},
        date={2005},
        ISSN={0025-5831},
     journal={Math. Ann.},
      volume={331},
      number={4},
       pages={795\ndash 807},
         url={https://doi.org/10.1007/s00208-004-0605-3},
      review={\MR{2148797}},
}

\bib{CLY}{article}{
      author={Cheng, Siu~Yuen},
      author={Li, Peter},
      author={Yau, Shing~Tung},
       title={On the upper estimate of the heat kernel of a complete
  {R}iemannian manifold},
        date={1981},
        ISSN={0002-9327},
     journal={Amer. J. Math.},
      volume={103},
      number={5},
       pages={1021\ndash 1063},
         url={https://doi.org/10.2307/2374257},
      review={\MR{630777}},
}

\bib{CLN}{book}{
      author={Chow, Bennett},
      author={Lu, Peng},
      author={Ni, Lei},
       title={Hamilton's {R}icci flow},
      series={Graduate Studies in Mathematics},
   publisher={American Mathematical Society, Providence, RI; Science Press, New
  York},
        date={2006},
      volume={77},
        ISBN={978-0-8218-4231-7; 0-8218-4231-5},
         url={http://dx.doi.org/10.1090/gsm/077},
      review={\MR{2274812 (2008a:53068)}},
}

\bib{Conway78}{book}{
      author={Conway, John~B.},
       title={Functions of one complex variable},
     edition={Second},
      series={Graduate Texts in Mathematics},
   publisher={Springer-Verlag, New York-Berlin},
        date={1978},
      volume={11},
        ISBN={0-387-90328-3},
      review={\MR{503901}},
}

\bib{GaLin86}{article}{
      author={Garofalo, Nicola},
      author={Lin, Fang-Hua},
       title={Monotonicity properties of variational integrals, {$A_p$} weights
  and unique continuation},
        date={1986},
        ISSN={0022-2518},
     journal={Indiana Univ. Math. J.},
      volume={35},
      number={2},
       pages={245\ndash 268},
         url={https://doi.org/10.1512/iumj.1986.35.35015},
      review={\MR{833393}},
}

\bib{GaLin87}{article}{
      author={Garofalo, Nicola},
      author={Lin, Fang-Hua},
       title={Unique continuation for elliptic operators: a
  geometric-variational approach},
        date={1987},
        ISSN={0010-3640},
     journal={Comm. Pure Appl. Math.},
      volume={40},
      number={3},
       pages={347\ndash 366},
         url={https://doi.org/10.1002/cpa.3160400305},
      review={\MR{882069}},
}

\bib{Ham1}{article}{
      author={Hamilton, Richard~S.},
       title={A matrix {H}arnack estimate for the heat equation},
        date={1993},
        ISSN={1019-8385},
     journal={Comm. Anal. Geom.},
      volume={1},
      number={1},
       pages={113\ndash 126},
         url={https://doi.org/10.4310/CAG.1993.v1.n1.a6},
      review={\MR{1230276}},
}

\bib{HL}{article}{
      author={Han, Qing},
      author={Lin, Fang-Hua},
       title={Nodal sets of solutions of elliptic differential equation},
     journal={Book in prepartion},
}

\bib{Lin91}{article}{
      author={Lin, Fang-Hua},
       title={Nodal sets of solutions of elliptic and parabolic equations},
        date={1991},
        ISSN={0010-3640},
     journal={Comm. Pure Appl. Math.},
      volume={44},
      number={3},
       pages={287\ndash 308},
         url={https://doi.org/10.1002/cpa.3160440303},
      review={\MR{1090434}},
}

\bib{Log}{article}{
      author={Logunov, Alexander},
       title={Nodal sets of {L}aplace eigenfunctions: polynomial upper
  estimates of the {H}ausdorff measure},
        date={2018},
        ISSN={0003-486X},
     journal={Ann. of Math. (2)},
      volume={187},
      number={1},
       pages={221\ndash 239},
         url={https://doi.org/10.4007/annals.2018.187.1.4},
      review={\MR{3739231}},
}

\bib{Log16}{article}{
      author={Logunov, Alexander},
       title={Nodal sets of {L}aplace eigenfunctions: proof of {N}adirashvili's
  conjecture and of the lower bound in {Y}au's conjecture},
        date={2018},
        ISSN={0003-486X},
     journal={Ann. of Math. (2)},
      volume={187},
      number={1},
       pages={241\ndash 262},
         url={https://doi.org/10.4007/annals.2018.187.1.5},
      review={\MR{3739232}},
}

\bib{Dan13}{article}{
      author={Mangoubi, Dan},
       title={The effect of curvature on convexity properties of harmonic
  functions and eigenfunctions},
        date={2013},
        ISSN={0024-6107},
     journal={J. Lond. Math. Soc. (2)},
      volume={87},
      number={3},
       pages={645\ndash 662},
         url={https://doi.org/10.1112/jlms/jds067},
      review={\MR{3073669}},
}

\bib{Ni07}{article}{
      author={Ni, Lei},
       title={A matrix {L}i-{Y}au-{H}amilton estimate for {K}\"ahler-{R}icci
  flow},
        date={2007},
        ISSN={0022-040X},
     journal={J. Differential Geom.},
      volume={75},
      number={2},
       pages={303\ndash 358},
         url={http://projecteuclid.org/euclid.jdg/1175266268},
      review={\MR{2286824}},
}

\bib{Ni15}{incollection}{
      author={Ni, Lei},
       title={Parabolic frequency monotonicity and a theorem of
  {H}ardy-{P}\'olya-{S}zeg\"o},
        date={2015},
   booktitle={Analysis, complex geometry, and mathematical physics: in honor of
  {D}uong {H}. {P}hong},
      series={Contemp. Math.},
      volume={644},
   publisher={Amer. Math. Soc., Providence, RI},
       pages={203\ndash 210},
         url={https://doi.org/10.1090/conm/644/12779},
      review={\MR{3372467}},
}

\bib{Po96}{article}{
      author={Poon, Chi-Cheung},
       title={Unique continuation for parabolic equations},
        date={1996},
        ISSN={0360-5302},
     journal={Comm. Partial Differential Equations},
      volume={21},
      number={3-4},
       pages={521\ndash 539},
         url={https://doi.org/10.1080/03605309608821195},
      review={\MR{1387458}},
}

\bib{Ste08}{incollection}{
      author={Zelditch, Steve},
       title={Local and global analysis of eigenfunctions on {R}iemannian
  manifolds},
        date={2008},
   booktitle={Handbook of geometric analysis. {N}o. 1},
      series={Adv. Lect. Math. (ALM)},
      volume={7},
   publisher={Int. Press, Somerville, MA},
       pages={545\ndash 658},
      review={\MR{2483375}},
}

\end{biblist}
\end{bibdiv}

\end{document}